\documentclass[a4paper]{article}

\usepackage[english]{babel}
\usepackage[utf8x]{inputenc}
\usepackage[T1]{fontenc}

\usepackage[a4paper,top=3cm,bottom=2cm,left=3cm,right=3cm,marginparwidth=1.75cm]{geometry}

\usepackage{amsmath,amssymb,tikz-cd}
\usepackage{graphicx}
\usepackage{url}
\usepackage[colorinlistoftodos]{todonotes}
\usepackage[colorlinks=true, allcolors=blue]{hyperref}
\usepackage{authblk}
\usepackage{xcolor}
\usepackage{mathrsfs}
\usepackage[english]{babel}
\newtheorem{definition}{Definition}
\newtheorem{theorem}{Theorem}
\newtheorem{proof}{Proof}
\newtheorem{remark}{Remark}
\numberwithin{equation}{subsection}

\title{\textbf{Causal Network Condensation}}
\author[1]{Lu Xuexing\thanks{xxlu@mail.ustc.edu.cn}}
\author[1]{Tang Ziqian\thanks{tangziqian@pku.edu.cn}}
\affil[1]{School of Mathematical Sciences, Zaozhuang University, Zaozhuang, P. R. China}
\affil[1]{School of Physics, Peking University, No.5 Yiheyuan Rd, Beijing 100871, P. R. China}

\begin{document}
\maketitle

\begin{abstract}
In this paper, we introduce a generalized topological quantum field theory based on the symmetric monoidal category, which we call causal network condensation since it can be regarded as a generalization of Baez's spin network construction, Turaev-Viro model and causal set theory. In this paper we introduce some new concepts, including causal network category, causal network diagrams, and their gauge transformations. Based on these concepts, we introduce the symmetric monoidal nerve functor for the symmetric monoidal category. It can be regarded as a generalization of the nerve functor for simplicial category. 
\end{abstract}

\section{Introduction}
John C. Baez [1] introduced the notion of generalized measures in gauge theory and spin network states as attempts to get rid of the difficulty of defining the measure on the space $\mathcal{A}/\mathcal{G}$ of connections modulo gauge transformations and to give the concept of quantum 3-geometries in loop quantum gravity. In his construction, he begins by studying gauge theory on a discrete directed graph $\phi=(V,E,s,t)$, defining connection and gauge transformations on $\phi$ and then the space $\mathcal{A}_{\phi}/\mathcal{G}_{\phi}$. Due to the discrete nature of the graph $\phi$, $\mathcal{A}_{\phi}/\mathcal{G}_{\phi}$ is isomorphic to a finite dimensional compact lie group and thus have a well-defined Haar measure. Applying Peter-Weyl theorem for it, we shall find a set of bases for $L^{2}(\mathcal{A}_{\phi}/\mathcal{G}_{\phi})$, each of which can be understood as a spin network on $\phi$. He then defines the embedding of the discrete graph $\phi$ to a given target manifold $M$. All graph embedding in $M$ and their refinement form a category. Since for each graph $\phi$ there is a Hilbert space spanned by all of its spin networks, we can take the projective limit to define $L^{2}(\mathcal{A}/\mathcal{G})$ on $M$ by its functoriality in a sense of generalized measure.
\\\\
In the usual content of LQG, we will then introduce so-called spin foams[3] as a description to quantum 4-geometries. Just as a spin network is a graph with edges labelled by representations and vertices labelled by intertwining operators, a spin foam is a 2-dimensional complex with faces labeled by representations and edges labelled by intertwining operators. Spin foams arise naturally as higher-dimensional analogues of Feynman diagrams foliated by spin networks in quantum gravity. In this sense, LQG becomes a extended topological quantum field theory similar to the Turaev-Viro model. And the spin foam category is analogous to the cobordism category.
\\\\
In this paper, we take a different route to understanding quantum 4-geometry than that of loop quantum gravity, inspired by the Ehlers-Pirani-Schild approach and causal set theory, with causal structure as its fundamental dynamical degree of freedom, but with features of both of these theories. To begin with, let recall the Ehlers-Pirani-Schild approach briefly.
\\\\
In early 70s, EPS[4][5] started from a set of well motivated physical properties of light rays and matter worldlines in a relativistic framework to derive the geometric structure of space-time from potentially observable objects. It is particularly suitable to discuss which geometric structures are observable and which are conventional. To be precise, EPS shows that one can obtain a class of Lorentzian metrics from lightrays. In this approach, spacetime is in a secondary position and should be derived from the matching of causal and projective structures as an emergence phenomenon.
\\\\
Similarly, for the theory of causal sets[6]. Its founding principles are that spacetime is formed by causal structure i.e. spacetime events are related by a partial order. This partial order has the physical meaning of the causality relations between spacetime events.
\\\\
This gives us a strong implication that we need to take causal structure as our fundamental degree of freedom rather than the holonomies in loop quantum gravity. But this is not enough; instead understanding spacetime volume as a naive probability measure as in causal set theory, we need a rich and robust algebraic structure to colour our causal structure to give both the space-time volume and the standard model of particle physics. Based on this, we will choose a symmetric monoidal category to colour the causal structures, as in the case of the $SU(2)$-spin network in loop quantum gravity, which will provide the appropriate richness and a robust mathematical structure for our definition of quantum spacetime consistent with the EPS approach.
\\\\
The plan of the paper is as follows. 
\\\\In Section 2 we give a precise definition of $\textit{causal network category}$ . It turns out that there is a category whose objects are causal relations (as a acyclic directed graphs and we call it causal networks) and whose morphisms are refinement and coarse-graining of causal relations (as functors between path category of graphs). It is show that any morphism in the causal network category can be decomposed as a composite of a series of elementary operations which we called $\textit{elementary morphism}$.
\\\\In Section 3, We will briefly review the concept of symmetric monoidal categories and define valuation on causal networks, which are similar to spin valuation for spin networks. Causal networks together with specific valuations are called $\textit{causal diagrams}$, and in order to make the functionality of them, we will introduce the notion of $\textit{gauge equivalence}$ of causal diagrams. 
\\\\In Section 4 We will explore its relation to the Baez construction and the nerve theorem and give a monoidal categorical version of the nerve theorem.
\\\\Section 5 is devoted to discussion and conclusions.
\section{The Category of Causal Networks}
The category of causal networks attempts to capture the most salient features of causal structure, with transformations between different causal networks analogous to the transformation of reference frame of observers. Based on this, we define a causal network as a directed acyclic graph, and a morphism between causal networks is defined as a refinement or coarse graining of the graph. The precise definition  is as follows. 
\begin{definition}
Causal network category is a category with whose objects are acyclic directed graphs (causal networks) $\phi = (V,E,s,t)$ and whose morphisms are functors between path category (free category) of directed graphs $f:F(\phi)\rightarrow F(\phi')$. The path category of directed graph $\phi$ is denoted by $F(\phi)$ and the causal network category is denoted by $\mathsf{Cau}$.
\end{definition}
\begin{definition}
A morphism $f$ in $\mathsf{Cau}$ is called $\textit{elementary morphism}$ if it can be induced by one of the following operations of directed graphs or isometry between directed graphs:\\
1. Adding an vertex $v$ to a directed graph $\phi$\\
2. Adding an edge e between two vertices $v_1 ,v_2$ of a directed graph $\phi$\\
3. Subdividing an edge $e$ of a directed graph $\phi$ by adding a vertex $v$ in the middle of it.\\
4. Coarse-graining of edges. For a set of edges $e_1 ,..., e_n$ in $\phi$ with a common source and target i.e. $s(e_1)=...=s(e_n),t(e_1)=...=t(e_n)$ we merge them into a single edge $e$.\\
5.Coarse-graining of vertices. For a subset $S$ of vertices of $\phi$, we denote the induced subgraph of $S$ in $\phi$ by $\phi (S)$. We shrink $\phi (S)$ to a single point.
\end{definition}
By "induced by one of the following operations of directed graphs" we mean that one of the following operations determines a unique functor between path category of the origin graph and the changed one. Since a functor between path category of directed graphs is determined by its image for each edge we can determine the functor by listing its image of edges.
\\\\
For isometry between directed graphs, this maps edges to corresponding edges determined by isomorphism. For operation 1 or 2, it is an inclusion map of edges. For operation 3, it maps edges to corresponding edges except the one being subdivided and for this one it maps to the path after subdividing. For operation 4, it maps the edges being merged to the merged one. For operation 5, it maps the edges and vertices to the shrinked point. It is shown that any morphism in the causal network category can be decomposed as a composite of a series of elementary morphism and we have the following theorem.
\\
\begin{theorem}
Given two causal networks $\psi$ and $\phi$ in $\mathsf{Cau}$, $f$ is a morphism between them, then $f$ can be decomposed into a composition of a series of elementary morphisms, i.e. $f=f_n \circ ... \circ f_1$ for some elementary morphisms $f_n \circ ... \circ f_1$
\end{theorem}
To proof this, we shall notice first that if $\phi$ is a acyclic then it will remain acyclic after shrink some of its disjoint subgraphs. This can be proved by contradiction. Suppose $\phi$ produces a loop after shrinking one of its subgraphs $G_i$ to a point $p_i$ , then there are two possibilities, the first
is the loop does not include the point $p_i$ that $G_i$ becomes after shrinking which implies that there existed a loop before shrinking, which contradicts the fact that there was no loop before shrinking. The second one is that the loop contains the point $p_i$ that $G_i$ becomes after its shrinking, then since there is ultimately acyclic, this loop must have been shrunk afterwards, so as $p_i$ which implies $V_i \cap V_j \ne \emptyset$ for some $i,j$ contradicts the fact that those $V_i$ are disjoint.
\begin{proof}
Given $f:F(\psi)\rightarrow F(\phi)$,we prove it by a direct construction of decomposition of $f$ into a series of elementary morphisms $f_i$. 
\\
First, we check the image $f(e)$ of each edge $e$ of $\psi$ under $f$. If $f(e)$ is a path with intermediate nodes. We pull these nodes back on $\psi$ This corresponds to a number of operations of subdivision. We will notate the diagram after the operation as $\psi'$.
\\
Then we check the image $f(v)$ of each vertex $v$ of $\psi'$ under $f$. If the image of several vertices is the same. We will collect them into a group. Repeat the above operation to divide the vertex set $V'$ of $\psi'$ into a number of non-intersecting sums of subsets of vertices $G_1,...,G_n$. We shrink these $G_i (i=1,...,n)$ one by one, which is equivalent to performing a number of "coarse-graining of vertices". We will notate the diagram after the operation as $\psi''$.
\\
We check the image $f(e)$ of each edge $e$ remains in $\psi''$ of $\psi'$(and also $\psi$), If some of them merge as an edge, we merge them as well which is equivalent to performing a number of "coarse-graining of edges". We will notate the diagram after the operation as $\psi'''$.
\\
$\psi'''$ can be seemed as a inclusion of $\phi$ and we add these additional edges and points by several operation 1,2,3.
\\
This process can be summarised in the following diagram (The five kinds of operations will be denoted by their number above).
\begin{center}
\begin{tikzcd}
\psi \arrow[r, "3"] & \psi' \arrow[r, "5"] & \psi'' \arrow[r, "4"] & \psi''' \arrow[r, "123"] & \phi
\end{tikzcd}
\end{center}
\end{proof}
It is easy to verify that the above operations are well defined.
\\
\begin{remark}
It is worth noting that the causal network category $\mathsf{Cau}$ is a generalization of the category of posets $\mathsf{Pos}$. For any partial order set $(S,\leq)$, there is a Hasse diagram for it. If one represents the edges of the poset's Hasse diagram as arcs in the upward direction, and take the transitive closure of the resulting (acyclic) digraph, he gets the directed acyclic graph $\psi(S,\leq)$ of the poset. Also, any isotone map $f$ between two partial order set $(S,\leq)$ and $(S',\leq)$ induces a functor $F(f)$ between the path categorys of $\psi(S,\leq)$ and $\psi(S',\leq)$. This defines an inclusion $i:\mathsf{Pos}\hookrightarrow \mathsf{Cau}$. Conversely, any directed acyclic graph $\psi$ defines a poset with its vertices and directed paths. Therefore, it is not hard to verify that $\mathsf{Pos}$ is a reflective subcategory of $\mathsf{Cau}$.
\end{remark}
\section{Causal digrams and their gauge equivalence}
\begin{definition}
Let$(\mathcal{V},\otimes,I,\alpha,l,r,B)$be a strict symmetric monoidal category and $\psi=(V,E,s,t)$ an acyclic directed graph. A causal diagram of $\psi$ in $\mathcal{V}$ is a triple $(\mathfrak{p},\mathfrak{v})$ including the following data.
\\\\
1.polarization $\mathfrak{p}$: For each vertex $x\in V$, assign a full order $\mathfrak{p}_{{\rm in},x}:[1,|{\rm in}(x)|] \xrightarrow{\sim} {\rm in}(x) ,\mathfrak{p}_{{\rm out},x}:[1,|{\rm out}(x)|] \xrightarrow{\sim} {\rm out}(x) $ to the input ${\rm in}(x):=t^{-1}(x) $ and output ${\rm in}(x):=s^{-1}(x)$ of $x$, respectively.
\\\\
2.valuation $\mathfrak{v}$: It contains two map $\mathfrak{v}_{0}:E \rightarrow {\rm obj}\mathcal{V},\mathfrak{v}_{1}:V \rightarrow {\rm mor}\mathcal{V}$ which are called edge valuation and vertex valuation respectively. 
\\\\
Edge valuation assigns each edge $e\in E$ an object $\mathfrak{v}_0 (e)\in {\rm mor}\mathcal{V}$ while vertex valuation assigns each vertex $x\in \mathcal{V}$a morphism $\otimes_{k=1}^{|{\rm in}(x)|}\mathfrak{p}_{{\rm in},x}(k)\xrightarrow{\mathfrak{v}_{1}(x)} \otimes_{k=1}^{|{\rm out}(x)|}\mathfrak{p}_{{\rm out},x}(k) $. We specify ${\rm dom}\mathfrak{v}_{1}(x)$ or ${\rm cod}\mathfrak{v}_{1}(x)$ as unit object $I$ in $\mathcal{V}$ when ${\rm in}(x)$ or ${\rm out}(x)$are empty sets.
\end{definition}
\begin{definition}
Let $(\mathfrak{p},\mathfrak{v})$ and $(\mathfrak{p}',\mathfrak{v}')$ be two causal diagrams of $\psi$ in $\mathcal{V}$. We say that the two diagrams are gauge equivalent if for each edge $e$ of $\psi$ there exists a isomorphism $f_e$ from $\mathfrak{v}_{0}(e)$ to $\mathfrak{v'}_{0}(e)$such that for each vertex $x$ in $\psi$,the following diagram commutes
\begin{equation}
\begin{tikzcd} [column sep = 15em]
\otimes_{k=1}^{|{\rm in}(x)|}\mathfrak{v}_{0}(\mathfrak{p}_{{\rm in},x}(k)) \arrow[r,"\langle\mathfrak{p'}_{{\rm in},x}^{-1}\circ\mathfrak{p}_{{\rm in},x}\rangle\circ (\otimes_{k=1}^{|{\rm in}(x)|}f_{\mathfrak{p}_{{\rm in},x}(k)})"] \arrow[d,"\mathfrak{v}_{1}(x)"]
& \otimes_{k=1}^{|{\rm in}(x)|}\mathfrak{v'}_{0}(\mathfrak{p'}_{{\rm in},x}(k))  \arrow[d, "\mathfrak{v'}_{1}(x)"] \\ \otimes_{k=1}^{|{\rm out}(x)|}\mathfrak{v}_{0}(\mathfrak{p}_{{\rm out},x}(k))\arrow[r, "\langle\mathfrak{p'}_{{\rm out},x}^{-1}\circ\mathfrak{p}_{{\rm out},x}\rangle\circ (\otimes_{k=1}^{|{\rm out}(x)|}f_{\mathfrak{p}_{{\rm out},x}(k)})"]
& \otimes_{k=1}^{|{\rm out}(x)|}\mathfrak{v'}_{0}(\mathfrak{p'}_{{\rm out},x}(k))
\end{tikzcd}
\end{equation}
Similarly, when ${\rm in}(x)$ or ${\rm out}(x)$ is an empty set, we then specify the corresponding horizontal morphism as $I\xrightarrow{{\rm id}_{I}} I$.
\end{definition}
\begin{theorem}
Denote the set of all causal diagrams of $\psi$ in $\mathcal{V}$ by ${\rm CauDiag}(\psi)$, then the gauge equivalence between causal diagrams is an equivalence relation. Let the quotient set of ${\rm CauDiag}(\psi)$ under this equivalence relation be ${\rm CD}(\psi)$
\end{theorem}
\begin{proof}
For two identical causal diagrams $(\mathfrak{p},\mathfrak{v})$, the $f_e$ is the identity morphism ${\rm id}_{\mathfrak{v}_0 (e)}$. For two causal diagrams $(\mathfrak{p},\mathfrak{v})$ and $(\mathfrak{p}',\mathfrak{v}')$ if $(\mathfrak{p},\mathfrak{v})$ is gauge equivalent to $(\mathfrak{p}',\mathfrak{v}')$ by $f_e$ it is easy to verify that $(\mathfrak{p}',\mathfrak{v}')$ is gauge equivalent to $(\mathfrak{p},\mathfrak{v})$ by $f^{-1}_e$. If for three causal diagrams $(\mathfrak{p},\mathfrak{v})$,$(\mathfrak{p}',\mathfrak{v}')$,$(\mathfrak{p}'',\mathfrak{v}'')$, $(\mathfrak{p},\mathfrak{v})$ is gauge equivalent to $(\mathfrak{p}',\mathfrak{v}')$ by $f_e$ and $(\mathfrak{p'},\mathfrak{v'})$ is gauge equivalent to $(\mathfrak{p}'',\mathfrak{v}'')$ by $f'_e$ it is easy to verify that $(\mathfrak{p},\mathfrak{v})$ is gauge equivalent to $(\mathfrak{p}'',\mathfrak{v}'')$ by $f'_e \circ f_e$.
\end{proof}
\section{Symmetric Monoidal Nerve Functor}
\subsection{Value of a valuation}
Joyal and Street [7] defined the value of a valuation on a polarised progressive graph (PPG) in a symmetric monoidal category $\mathcal{V}$. The data required to define the value of a valuation are as follows:
\\\\
1. A polarised progressive graph $\Gamma$ with ${\rm dom}$ and ${\rm cod}$;\\
2. A valuation of $\Gamma$ in $\mathcal{V}$;\\
3. Linear order on ${\rm dom}$ and ${\rm cod}$.
\\\\
Given the data above we can define a morphsim $\mathfrak{v}(\Gamma)\in{\rm mor}\mathcal{V}$, called the value of this valuation. We claim that, in our setting, a causal network $\psi=(V,E,s,t)$together with a causal diagram $(\mathfrak{p},\mathfrak{v})$ and a subset $S\subset V$ define a PPG with a valuation in the sense of [7] with whose set of edges, set of innner nodes, ${\rm dom}$ and ${\rm cod}$ are $\{e\in E|s(e)\in S\ {\rm or}\  t(e)\in S\},S,{\rm dom}=\{e\in E|s(e)\ne S,t(e)\in S\},{\rm cod}=\{e\in E|s(e)\in S,t(e)\ne S\}$ respectively and whose orientation,polarization and valuation inherits that of $\psi$ and $(\mathfrak{p},\mathfrak{v})$.Further, fixing two linear order $\mathfrak{q}_{\rm dom},\mathfrak{q}_{\rm cod}$ on ${\rm dom}$ and ${\rm cod}$ gives the value in [7] which we denote it as $\mathfrak{v}(\psi,\mathfrak{p},\mathfrak{v},S,\mathfrak{q}_{\rm dom},\mathfrak{q}_{\rm cod})$.In particular, if ${\rm dom}$ or ${\rm cod}$ is the empty set, we specify that the domain or codomain of the value is the unit object $I$ in $\mathcal{V}$.
\subsection{Symmetric monoidal nerve of a symmetric monoidal category}
Let $f$ be an elementary morphism between two casual network $\psi$ and $\psi'$. Then $f$ can induce a map $F[f]$ from the set ${\rm CD}(\psi)$ to ${\rm CD}(\psi')$ as follows.
\\\\1.If $\psi$ is isomorphic to $\psi'$, then for any causal diagram $(\mathfrak{p},\mathfrak{v})$ on $\psi$, we can pass the polarization and valuation on $\psi$ to $\psi'$ by isomorphism to give a causal diagram $(\mathfrak{p}',\mathfrak{v}')$ on $\psi'$.
\\\\2.If $\psi'$ can be obtained by adding a vertex $x$ to $\psi$, then for any causal diagram $(\mathfrak{p},\mathfrak{v})$ on $\psi$, the causal diagram $(\mathfrak{p}',\mathfrak{v}')$ on $\psi'$ is identical to $(\mathfrak{p},\mathfrak{v})$, except that the point $x$ is assigned with $I\xrightarrow{{\rm id}_{I}} I$.
\\\\3.If $\psi'$ can be obtained by adding a directed edge $e$ from two vertices $x_{1}$ and $x_{2}$ of $\psi$. For any $[(\mathfrak{p},\mathfrak{v})]\in{\rm CD}(\psi)$, we specify that $[(\mathfrak{p'},\mathfrak{v'})]\in{\rm CD}(\psi')$ is a causal diagram where $(\mathfrak{p'},\mathfrak{v'})$ keep the polarization unchanged except for ${\rm out}(x_{1}),{\rm in}(x_{2})$. We add $e$ to ${\rm out}(x_{1}),{\rm in}(x_{2})$ as its minimal element, leaving the order of the remaining elements unchanged and assign $I$ to $e$.This is well defined since gauge equivalence.
\\\\4.If $\psi'$ can be obtained by adding a vertex $x$ to a directed edge $e$ of $\psi$ so that it divides into two directed edges $e_{1}$ and $e_{2}$, then for any causal diagram $(\mathfrak{p},\mathfrak{v})$ on $\psi$,the causal diagram $(\mathfrak{p}',\mathfrak{v}')$ on $\psi'$ is identical to $(\mathfrak{p},\mathfrak{v})$($e_1$ and $e_2$ take the place of $e$ in the linear order), except that the edges $e_{1}$ and $e_{2}$ and the vertex $x$. We assign the data on the original edge $e$ to edges $e_1,e_2$ and ${\rm id}_{\mathfrak{v}_0 (e)}$to vertex $x$.
\\\\5.If one can obtain $\psi'$ by merge several directed edges $e_{1},...,e_{n}$ with a common source $x_{1}$ and target $x_{2}$ into a single directed edge $e$ pointing from $x_{1}$ to $x_{2}$. Then for any $[(\mathfrak{p},\mathfrak{v})]\in{\rm CD}(\psi)$, we pick a representative causal diagram $(\mathfrak{p},\mathfrak{v})$ and specify that $[(\mathfrak{p'},\mathfrak{v'})]\in{\rm CD}(\psi')$ is a causal diagram identical to $[(\mathfrak{p},\mathfrak{v})]\in{\rm CD}(\psi)$ except that the edges $e_{1},...,e_{n}$. Data on the edge $e$ is the tensor product of $\mathfrak{v}_{0}(e_1),...,\mathfrak{v}_{0}(e_n)$ (after appropriate permutation such that $e_1,...,e_n$ are adjacent to each other and of the same order on both sides). This is well defined since gauge equivalence.
\\\\6.If $\psi'$ can be obtained by shrinking the induced subgraph $\psi(S)$ of a subset $S$ of vertices of $\psi$ to a point $x$, then for any causal diagram $(\mathfrak{p},\mathfrak{v})$ on $\psi$ the causal diagram $(\mathfrak{p}',\mathfrak{v}')$ on $\psi'$ is identical to $(\mathfrak{p},\mathfrak{v})$ except the data on the induced subgraph $\psi(S)$.We assign $\mathfrak{v}(\psi,\mathfrak{p},\mathfrak{v},S,\mathfrak{q}_{\rm dom},\mathfrak{q}_{\rm cod})$ to $x$. It should be noticed that $\mathfrak{q}_{\rm dom},\mathfrak{q}_{\rm cod}$ can be specified arbitrarily since gauge equivalence.
\\\\
By Theorem 1, for any two acyclic directed graphs $\psi,\phi$, a morphism $f$ between them can be decomposed into a composite of several elementary morphisms $f_1,...,f_n$. Together with the above discussion, $f$ defines a map from set ${\rm CD}(\psi)$ to ${\rm CD}(\phi)$. This is obviously independent of the way $f$ is decomposed and hence functorial. We conclude these as the following.
\begin{definition}
The operations above defines a functor $p\mathcal{SN}[\mathcal{V}]:\mathsf{Cau}\rightarrow \mathsf{Set}$ which we called the $\textit{Pre symmetric monoidal nerve functor}$ of the symmetric monoidal category $\mathcal{V}$. Furthermore, it defines a functor $p\mathcal{SN}:\mathsf{SymMon}\rightarrow \mathsf{CoPsh}(\mathsf{Cau})$ from the category of symmetric monoidal categories to the category of presheaves over $\mathsf{Cau}$.
\end{definition}
However the functor $p\mathcal{SN}:\mathsf{SymMon}\rightarrow \mathsf{CoPsh}(\mathsf{Cau})$ is not fully-faith. We consider that there exists a unique fully-faith functor $\mathcal{SN}:\mathsf{SymMon}\rightarrow \mathsf{CoPsh}(\mathsf{Cau})$ satisfies certain rigid conditions such that $p\mathcal{SN}$ acts as its subfunctor and we call it the $\textit{symmetric monoidal nerve functor}$ . We assume in the following that $\mathcal{SN}$ exists.
\begin{remark}
The construction of symmetric monoidal nerve functor is motivated by the notion of the nerve $N[\mathcal{C}]$ of a category $\mathcal{C}$ which is a geometric realization of categories. Analogously to this, we note that the monoidal product $\otimes$ of the monoidal category has a dual relationship with the graphical calculus of the string diagram, so the graphical calculus can be used to test the nonlinear structure within the monoidal category. For this reason, we introduce the symmetric monoidal nerve functor, which is a natural geometric realization of monoidal category.
\end{remark}
\section{Casual Network Condensation}
\subsection{Category of paths}
The following requires the categorification of some geometric objects (manifolds or directed graphs). The basic idea is to identify the corresponding geometrical objects to their "category of paths".
\\\\
Given a smooth manifold $M$, let $\mathcal{P}(M)$ be the path groupoid of $M$. Given two smooth manifold $M$ and $M'$, a smooth diffeomorphism $f:M\rightarrow M'$ defines a functor $\mathcal{P}(f):\mathcal{P}(M)\rightarrow \mathcal{P}(M')$ sending each point $p\in M$ to $f(p)\in M'$ and each path $\gamma: [0,1]\rightarrow M$ with $\gamma [0]=p, \gamma [1]=q$ to $f\circ\gamma:[0,1]\rightarrow M'$. This defines a functor $\iota_1:\mathsf{SmthMfd}\rightarrow \mathsf{Cat}$ from the category of smooth manifold to the category small categories sending each smooth manifold $M$ to its path groupoid $\mathcal{P}(M)$ and each smooth diffeomorphism $f$ to $\mathcal{P}(f)$.
\\\\
Let $\mathsf{dGrph}$ be the category of directed graphs with whose objects are directed graphs and whose morphisms are inclusion of graphs defined in [2]. Each directed graph $\psi$ in $\mathsf{dGrph}$ can be identified with a category $\mathcal{P}(\psi)$ with whose objects are vertices of $\psi$ and whose morphism are undirected paths connecting vertices. Further, an inclusion $i:\psi\rightarrow\phi$ defines a functor $\mathcal{P}(i):\mathcal{P}(\psi)\rightarrow\mathcal{P}(\phi)$ by sending vertices and paths in $\mathcal{P}(\psi)$ to those in $\mathcal{P}(\phi)$ defined by $i$. This defines a functor $\iota_2:\mathsf{dGrph}\rightarrow \mathsf{Cat}$ from the category of smooth manifold to the category small categories sending each directed graph $\psi$ to $\mathcal{P}(\psi)$  and each inclusion $i$ to $\mathcal{P}(i)$.
\\\\
Finally, we can define a functor $\iota_3 :\mathsf{Cau}\rightarrow\mathsf{Cat}$ sending each directed graph $\psi$ to its path category $F(\psi)$ and each morphism $F(f)$ in $\mathsf{Cau}$ to itself.
\begin{remark}
It is worth noting that $\mathsf{dGrph}$ differs from $\mathsf{Cau}$ in that the former's graph inclusion contains direction reversal of edges whereas the latter does not. This corresponds to the intuition that the order of cause and effect cannot be reversed.
\end{remark}
\subsection{Casual Network Condensation}
Baez's generalized measure [2] on the space $L^2 (\mathcal{A} / \mathcal{G})$ with a base manifold $M$ and gauge group $G$ can be described by the following commutative diagram
\\
\begin{center}
\begin{tikzcd}[column sep=scriptsize]
\mathsf{dGrph} \arrow[dr,"\iota_2"] \arrow[rr, "\quad"{name=U, below},"\mathcal{S}"]{}
    & & \mathsf{Hilb}  \\
& \mathsf{Cat} \arrow[ur,dashed,"{\rm Lan}_{\iota_2}\mathcal{S}"']\arrow[Rightarrow, from=U, "\eta"] \\
& \mathsf{SmthMfd} \arrow[u,"\iota_1"]
\end{tikzcd}
\end{center}.
\\
where $\mathcal{S}$ is the functor the send each $\psi$ in $\mathsf{dGrph}$ to the space $L^2(\mathcal{A}_\psi /\mathcal{G}_\psi)$ with respect to the canonical measure on the space $\mathcal{A_\psi}/\mathcal{G_\psi}$ of smooth connections modulo gauge transformations on a directed graph $\psi$. Given a manifold $M$, the Hibert space $L^2(A/G)$ with respect to the canonical generalized measure on the space $\mathcal{A}/\mathcal{G}$ of smooth connections on $P$ modulo gauge transformations defined in [2] is exactly ${\rm Lan}_{\iota_2}\mathcal{S}\circ\iota_1(M)$.This reminds us of the formal structure of factorization algebra[8] as follows
\begin{center}
\begin{tikzcd}
\mathcal{D}\mathsf{isk}^\sqcup _{n,{\rm or}} \arrow[d, "\iota"', hook] \arrow[r, "A"] & \mathcal{C}^{\otimes} \\
\mathcal{M}\mathsf{fld}^\sqcup _{n,{\rm or}} \arrow[ru, "\int A"', dotted]                &     
\end{tikzcd}
\end{center}
It extends a local algebra of observables to a global algebra of observables by an Kan extension as well. In our case, the generalization is straightforward. Given a symmetric monoidal category $\mathcal{V}$ the disk algebra $\mathcal{D}\mathsf{isk}^\sqcup _{n,{\rm or}}$ is replaced by the category $\mathsf{Cau}$ which can be seem as a generalization of "open set" and the functor $A$ is replaced by the functor ${\rm Free} \circ \mathcal{SN}[\mathcal{V}]$ where ${\rm Free}:\mathsf{Set}\rightarrow \mathsf{Vect}$ is the functor of free generation and $\mathcal{SN}[\mathcal{V}]$ is the symmetric monoidal nerve Functor mentioned above. The image of it can be seemed as a generalization of "the set of functions on an open set". We call the following commutative diagram as "Causal Network condensation"
\\
\begin{center}
\begin{tikzcd}[column sep=scriptsize]
\mathsf{Cau} \arrow[dr,"\iota_3"] \arrow[rr, "\quad"{name=U, below},"\mathcal{SN}{\rm [}\mathcal{V}{\rm ]}"]{}
    & & \mathsf{Set}\\
& \mathsf{Cat} \arrow[ur,dashed,"{\rm Lan}_{\iota_3}\mathcal{SN}{\rm [}\mathcal{V}{\rm ]}"']\arrow[Rightarrow, from=U, "\eta'"] \\
& \mathsf{SmthMfd} \arrow[u,"\iota_1"]
\end{tikzcd}
\end{center}.
\\
The construction is exactly parallel to that of the factorization algebra and can be seen as a new type of homology. In physics, it generalized the construction of spin network and gives a quantum measure with universal applicability.
\section{Conclusions}
We first define the category of causal networks $\mathsf{Cau}$ and the notion of causal diagrams on it, and we show that the deformation of a causal network can naturally induce the deformation of a causal diagram on it and hence functorial. Further, it can give a construction parallel to the construction of factorization homology as they can both be described by Kan extension This can be seen as a quantum generalization of manifold or a generalized homology.
\\
In the framework of Causal network condensation. the concept of topological structure is replaced by causal networks and ring of functions is replaced by causal diagrams while coordinate transformation is replaced by the functor $\mathcal{SN}[\mathcal{V}]$.

\bibliographystyle{alpha}
\bibliography{sample}

\end{document}